\documentclass[12pt,a4paper,fleqn,leqno]{amsart}
\usepackage{amssymb}
\usepackage{mathrsfs}
\usepackage[all]{xy}
\usepackage{fancyhdr}
\usepackage[pdftex,bookmarks,colorlinks]{hyperref}

\theoremstyle{plain}
\newtheorem{theorem}{Theorem}[section]
\newtheorem{lemma}[theorem]{Lemma}
\newtheorem{corollary}[theorem]{Corollary}
\newtheorem{proposition}[theorem]{Proposition}

\newtheorem{question}[theorem]{Question}

\theoremstyle{definition}

\newtheorem{remark}[theorem]{Remark}

\renewcommand{\emptyset}{\varnothing}

\DeclareMathOperator{\uhr}{\upharpoonright}

\DeclareMathOperator{\N}{\mathbb{N}}

\DeclareMathOperator{\R}{\mathbb{R}}

\DeclareMathOperator{\conv}{conv}

\numberwithin{equation}{section}

\begin{document}

\title{Selections for Paraconvex-Valued Mappings on non-paracompact domains}

\thanks{The work of the author is based upon research supported by the German
Academic Exchange Service (DAAD)}

\author{Narcisse Roland Loufouma Makala}

\address{School of Mathematical Sciences, University of KwaZulu-Natal,
  Westville Campus, Private Bag X54001, Durban 4000, South Africa}

\email{roland@aims.ac.za}



\begin{abstract}
   We prove that Michael's paraconvex-valued selection theorem
   for paracompact spaces remains true  for $\mathscr{C}'(E)$-valued
   mappings defined on collectionwise normal spaces. Some possible
   generalisations are also given.
\end{abstract}

\date{\today}
\maketitle

\section{Introduction}
\label{section1}

For a topological space $E$, let $2^E$ be the family of all nonempty
subsets of $E$, and $\mathscr{F}(E)$ be the subfamily of $2^E$
consisting of all closed members of $2^E$. A set-valued mapping
$\varphi:X\to 2^E$ is \emph{lower semi-continuous}, or
\emph{l.s.c.}, if the set
\[
\varphi^{-1}(U)=\{x\in X: \varphi(x)\cap U\neq \emptyset\}
\]
is open in $X$ for every open $U\subset E$. A set-valued mapping
$\psi:X\to 2^E$ is \emph{upper-semi continuous}, or \emph{u.s.c.},
if the set
\[
\psi^{\#}(U)=\{x\in X:\psi(x)\subset U\}
\]
is open in $X$ for every open $U\subset E$. Equivalently, $\psi$ is
u.s.c. if $\psi^{-1}(F)$ is closed in $X$ for every closed subset
$F\subset E$. A single-valued mapping $f:X\to E$ is a \emph{selection} for
$\varphi:X\to 2^E$ if $f(x)\in \varphi(x)$ for every $x\in X$.
\medskip

Let $E$ be a normed space. Throughout this paper, we will use $d$ to
denote the metric on $E$ generated by the norm of $E$. Following
\cite{michael2}, a subset $P$ of $E$ is called
$\alpha$-\emph{paraconvex}, where $0\leq\alpha\leq1$, if whenever $r>0$
and $d(p,P)<r$ for some $p\in E$, then
\[
d(q,P)\leq\alpha r\ \text{for all}\ q\in\conv(B_r(p)\cap P).
\]

Here, $B_r(x)=\{y\in E:d(x,y)<r\}$, and $\conv(A)$ is the
\emph{convex hull} of $A$. The set $P$ is called \emph{paraconvex}
if it is $\alpha$-paraconvex for some $\alpha<1$. A closed set is
$0$-paraconvex if and only if it is convex. In the sequel, we will
use $\mathscr{F}_{\alpha}(E)$ to denote all $\alpha$-paraconvex
members of $\mathscr{F}(E)$ (i.e., all nonempty closed
$\alpha$-paraconvex subsets of $E$).\medskip

Recall that a space $X$ is \emph{paracompact} if it is Hausdorff and
every open cover of $X$ has a locally finite open
refinement. In \cite{michael}, E. Michael proved that if $X$ is paracompact and
$E$ is a Banach space, then every l.s.c.\ convex-valued mapping
 $\varphi:X\to\mathscr{F}(E)$ has a continuous selection
(see, \cite[Theorem 3.2"]{michael}). In \cite{michael2}, E. Michael
generalised this result by replacing ``convexity" with
``$\alpha$-paraconvexity" for a fixed $\alpha<1$, and proved the
following theorem.

\begin{theorem}[\text{\cite[Theorem 2.1]{michael2}}]
\label{m-sel-thm}Let $X$ be a paracompact space, $E$ be a Banach
space, and let $\varphi:X \to\mathscr{F}_{\alpha}(E)$ be an l.s.c.\
mapping, where $\alpha<1$. Then, the following hold\textup{:}
\begin{itemize}
\item[(a)] $\varphi$ has a continuous selection.
\item[(b)]If $r>0$ and $g:X\to E$ is continuous such that
$d(g(x),\varphi(x))<r$ for all $x\in X$, then there exists
$\delta>0$ and a continuous selection $f$ for $\varphi$ such that
$d(g(x),f(x))<\delta r$, $x\in X$.
\end{itemize}
\end{theorem}

It should be remarked that Theorem \ref{m-sel-thm} is not true for
$\alpha=1$. Indeed, V. Klee \cite{klee} proved that every subset of
an inner-product space $H$ (in particular, of a Hilbert space $H$)
is $1$-paraconvex, while not every l.s.c. mapping $\varphi:X\to
\mathscr{F}_1(H)$, from a paracompact space $X$ has a continuous
selection, because in this case, $\mathscr{F}_1(H)=\mathscr{F}(H)$.
\medskip

Let us now state the main purpose of this paper. Namely, in Section
\ref{section2}, we prove a collectionwise normal version of Theorem
\ref{m-sel-thm} (see, Theorem \ref{main-theorem}), thus generalising
\cite[Theorem 3.2']{michael} (see also \cite{choban-valov}; for
alternative proofs, see \cite{gutev-loufouma,nedev}) in terms of
paraconvex sets. In Section \ref{section3}, we show how our arguments
can be used to generalise further some of these results.

\section{Collectionwise Normality, Paraconvexity and Selections}
\label{section2}

Recall that a $T_1$-space $X$ is \emph{$\tau$-collectionwise
normal}, where $\tau$ is an infinite cardinal number, if for every
discrete collection $\mathscr{D}$ of closed subsets of $X$, with
$|\mathscr{D}|\leq\tau$, there exists a discrete collection
$\big\{U_D: D\in \mathscr{D}\big\}$ of open subsets of $X$ such that
$D\subset U_D$ for every $D\in \mathscr{D}$. A space $X$ is
\emph{collectionwise normal} if it is $\tau$-collectionwise normal
for every $\tau$. It is well known that $X$ is normal if and only if
it is $\omega$-collectionwise normal. Clearly, collectionwise
normality lies between paracompactness and normality.\medskip

In what follows, for a space $E$, let $\mathscr{C}(E)=\{S\in\mathscr{F}(E)
:S\ \text{is compact}\}$, and $\mathscr{C}'(E)=\mathscr{C}(E)\cup \{E\}$.
Also, for a normed space $E$, we will use the subscript $\alpha$ to
denote all $\alpha$-paraconvex members of $\mathscr{C}(E)$ or
$\mathscr{C}'(E)$. Finally, $w(E)$ denotes the \emph{topological
weight} of $E$.

\begin{theorem}
\label{main-theorem}Let $X$ be a $\tau$-collectionwise normal space,
$E$ be a Banach space with $w(E)\leq\tau$, and $\varphi:X\to
\mathscr{C}'_{\alpha}(E)$ be an l.s.c.\ mapping, for some
$\alpha<1$. Then, the following hold:
\begin{itemize}
\item[(a)]$\varphi$ has a continuous selection.
\item[(b)]If $r>0$ and $g:X\to E$ is continuous such that
$d(g(x),\varphi(x))<r$ for all $x\in X$, then there exists
$\delta>0$ and a continuous selection $f$ for $\varphi$ such that
$d(g(x),f(x))<\delta r$, $x\in X$.
\end{itemize}
\end{theorem}

To prepare for the proof of Theorem \ref{main-theorem}, we need the
following proposition. In the proof of this proposition and what
follows, a set-valued mapping $\psi:X\to 2^E$ is a
\emph{multi-selection} (or, a \emph{set-valued selection}) for
another set-valued mapping $\varphi:X\to 2^E$ if
$\psi(x)\subset\varphi(x)$, for every $x\in X$.

\begin{proposition}
\label{prop1} Let $X$ be a $\tau$-collectionwise normal space, $E$
be a completely metrizable space with $w(E)\leq\tau$,
$\{V_n:n\in\N\}$ an increasing open cover of $E$, and
$\varphi:X\to\mathscr{C}'(E)$ an l.s.c.\ mapping. Then,
there exists an increasing closed cover $\{A_n:n\in\N\}$ of $X$ such
that $A_n\subset\varphi^{-1}(V_n)$, for every $n\in\N$.
\end{proposition}

\begin{proof}
Since $\{V_n:n\in\N\}$ is an increasing open cover of $E$ and $E$ is
normal and countably paracompact (being metrizable), there exists
an increasing closed cover $\{F_n:n\in\N\}$ of $E$ such that
$F_n\subset V_n$, for every $n\in\N$. We then have
\begin{equation}
\label{eq1} \varphi^{-1}(F_n)\subset\varphi^{-1}(V_n),\
\text{for every}\ n\in\N.
\end{equation}

By a result of Choban and Valov \cite{choban-valov} (see also Nedev
\cite{nedev}), there exists a u.s.c.\ multi-selection $\psi:X\to
\mathscr{C}(E)$ for $\varphi$. Since $\psi$ is u.s.c., each
$\psi^{-1}(F_n)$, $n\in\N$, is closed in $X$. Since
$\psi(x)\subset\varphi(x)$, $x\in X$, we have

\begin{align*}
\psi^{-1}(F_n)=&\{x\in X:\psi(x)\cap F_n\neq\emptyset\}\\
\subset&\{x\in X:\varphi(x)\cap F_n\neq\emptyset\}=
\varphi^{-1}(F_n).
\end{align*}

The last inclusion and \eqref{eq1} imply that the family
$\{A_n:n\in\N\}$, with $A_n=\psi^{-1}(F_n)$, is an increasing closed
cover of $X$ such that $A_n\subset\varphi^{-1}(V_n)$, for every
$n\in\N$.
\end{proof}

Recall that if $(E,d)$ is a metric space, a mapping $\psi:X\to 2^E$
is \emph{$d$-l.s.c.} (resp. \emph{$d$-u.s.c.}) if, given
$\varepsilon>0$, every $x\in X$ admits a neighborhood $V$ such that
$\psi(x)\subset B_{\varepsilon}(\psi(z))$ (resp. $\psi(z)\subset
B_{\varepsilon}(\psi(x))$), for every $z\in V$. A mapping $\psi$ is said to be
\emph{$d$-continuous} if it is both $d$-l.s.c.\ and $d$-u.s.c.; $\psi$
is \emph{continuous} if it is both l.s.c.\ and u.s.c.; and
$\psi$ is \emph{$d$-proximal continuous} if it is both
l.s.c.\ and $d$-u.s.c.\ (see, \cite{gutev}). Every $d$-continuous or
continuous mapping is $d$-proximal continuous, but the converse is not
true (see, for instance \cite[Proposition 2.5]{gutev}). The following
Lemma also plays an essential role in the proof of Theorem \ref{main-theorem}.

\begin{lemma}[\text{\cite[Lemma 4.2]{gutev-al}}]
\label{lemma1}Let $X$ be a $\tau$-collectionwise normal space,
$E$ be a Banach space with $w(E)\leq\tau$, $\psi:X\to\mathscr{F}(E)$
be a $d$-proximal continuous convex-valued mapping, and
$\varphi:X\to\mathscr{F}(E)$ be an l.s.c. convex-valued
multi-selection for $\psi$ such that  $\varphi(x)$ is compact
whenever $\varphi(x)\neq\psi(x)$, $x\in X$. Then, $\varphi$ has a
continuous selection.
\end{lemma}

\begin{proof}[Proof of Theorem \ref{main-theorem}]
Let $X$, $E$, $\alpha$, and $\varphi$ be as in that theorem. We
first prove (b), and then (a).\medskip

(b). Since $\alpha<1$, there exists $\gamma\in\R$ such that
$\alpha<\gamma<1$. Then, $\sum^\infty_{i=0}\gamma^i<\infty$ (i.e.
the series $\sum^\infty_{i=0}\gamma^i$ converges). So, take $\delta$
such that $\sum_{i=0}^{\infty}\gamma^i< \delta$. To show that this
$\delta$ works, by induction, we shall define a sequence of
continuous maps $f_n:X\to E$, $n<\omega$, such that for all $n$ and
all $x\in X$,

\begin{equation}
\label{eq2}
d(f_n(x),\varphi(x))<\gamma^nr,
\end{equation}
\begin{equation}
\label{eq3}
d(f_n(x),f_{n+1}(x))\leq\gamma^nr.
\end{equation}

This will be sufficient because by \eqref{eq3}, $\{f_n:n<\omega\}$ is
a Cauchy sequence in $E$ which is complete, so it must converge to
some continuous map $f:X\to E$. By \eqref{eq2}, $f(x)\in\varphi(x)$, for
every $x\in X$, and by \eqref{eq3}

\begin{align*}
d(g(x),f_{n+1}(x))&=d(f_0(x),f_{n+1}(x))\\
&\leq d(f_0(x),f_1(x))+d(f_1(x),f_2(x))+\cdots+d(f_n(x),f_{n+1}(x))\\
&\leq r+\gamma r+\gamma^2r+\gamma^3r+\cdots+\gamma^nr\\
&=r\sum_{i=0}^{n}\gamma^i.
\end{align*}

Therefore, $d(f(x),g(x))\leq r\sum_{i=0}^{\infty}\gamma^i<\delta
r$.\medskip

Let $f_0=g$, which satisfies \eqref{eq2}. Suppose that $f_n$ has
been constructed for some $n\geq0$, and let us construct $f_{n+1}$.
Define a mapping $\psi_{n+1}:X\to\mathscr{F}(E)$ by
$\psi_{n+1}(x)=\overline{B_{\gamma^nr}(f_n(x))}$, $x\in X$. Then,
$\psi_{n+1}$ is $d$-proximal continuous (being $d$-continuous) and
convex-valued. Define another mapping
$\varphi_{n+1}:X\to \mathscr{F}(E)$ by
\[
\varphi_{n+1}(x)=\overline{\conv(\varphi(x)\cap
B_{\gamma^nr}(f_n(x)))},\ x\in X.
\]

By the inductive assumption, $\varphi_{n+1}(x)$ is never empty for
every $x\in X$, because $f_n$ satisfies \eqref{eq2} above.
Furthermore, by \cite[Propositions 2.3, 2.5, and 2.6]{michael},
$\varphi_{n+1}$ is l.s.c.\ and it is clearly convex-valued. Finally,
$\varphi_{n+1}$ is a multi-selection of $\psi_{n+1}$ and
$\varphi_{n+1}(x)\neq\psi_{n+1}(x)$ implies that $\varphi_{n+1}(x)$
is compact. Then, by Lemma \ref{lemma1}, $\varphi_{n+1}$ has
a continuous selection $f_{n+1}(x):X\to E$ such that
\[
d(f_n(x),f_{n+1}(x))\leq\gamma^nr,
\]
which is \eqref{eq3}. Since $\varphi(x)$ is $\alpha$-paraconvex for every $x\in X$, and
$\alpha<\gamma$, we have
\[
d(f_{n+1}(x),\varphi(x))\leq\alpha\cdot\gamma^nr<\gamma^{n+1}r,\
\text{for all}\ x\in X,
\]
that is, $f_{n+1}$ satisfies also \eqref{eq2}.\medskip

(a). Take $\beta>1$ such that $\varphi(x)\cap B_{\beta}(0)\neq
\emptyset$ for some $x\in X$, where $0$ is the origin of $E$. Let
\[
V_n=B_{\beta^n}(0),\ \text{for each}\ n\in\N.
\]

Then, each $V_n$ is open in $E$, and the family $\{V_n:n\in\N\}$ is
an increasing open cover of $E$. By Proposition \ref{prop1}, there
exists an increasing closed cover $\{B_n:n\in\N\}$ of $X$ such that
$B_n\subset\varphi^{-1}(V_n)$, for every $n\in\N$. Since $X$ is
normal, there are open sets $U_n\subset X$ such that $B_n\subset U_n
\subset\overline{U_n}\subset\varphi^{-1}(V_n)$ and $U_n\subset
U_{n+1}$, for each $n\in\N$. Letting $A_n=\overline{U_n}$ and following
the construction in the proof of (a) in \cite[Theorem 2.1]{michael2},
we get (using (b) above) a sequence of continuous selections $f_n:A_n\to E$ for
$\varphi\uhr A_n$ such that $f_{n+1}\uhr A_n=f_n$, $n\in\N$. Then, the
mapping $f:X\to E$ defined by $f\uhr A_n=f_n$, $n\in\N$, is a selection
for $\varphi$ which is continuous because each $f\uhr U_n=f_n\uhr U_n$,
$n\in\N$, is continuous and $\{U_n:n\in\N\}$ is an open cover of $X$.
The proof is completed.
\end{proof}

\begin{remark}
As it was already mentioned, the proof of (a) in Theorem \ref{main-theorem} follows
the proof in Theorem \ref{m-sel-thm} (see, \cite[Theorem 2.1]{michael2}).
However, the proof in \cite{michael2} contains a minor gap where the sets $A_n$,
$n\in\N$, were only assumed to be closed rather than $A_n=\overline{U_n}$, for some
increasing open cover $\{U_n:n\in\N\}$ of $X$. If the condition $A_n=\overline{U_n}$,
$n\in\N$, is not explicitly required, then the resulting selection $f:X\to E$ defined
by $f\uhr A_n=f_n$, $n\in\N$, may fail to be continuous. An example of such a situation
is given by the function $f(x)=\sin(1/x)$, $0<x\leq1$, $f(0)=0$ and
$A_n=\{0\}\cup[n^{-1},1]$.
\end{remark}

\section{Some Possible Generalisations}
\label{section3}

By \cite[Corollary 2.2]{michael2}, if $X$ is paracompact, $A\subset
X$ is closed and $Y$ is a closed paraconvex subset of a Banach space
$E$, then every continuous $g:A\to Y$ can be extended to a
continuous $f:X\to Y$. According to Dowker's extension theorem
\cite{dowker}, this implies that the same remains valid for $X$
being only collectionwise normal. As a rule, the theorems for the
existence of continuous selections for l.s.c.\ mappings originated
as a natural generalisation of extension theorems, see Michael
\cite{michael,michael6}. In view of the above, this brings the
question for a more natural setting of Theorem \ref{main-theorem}.
Namely, given $0\leq\alpha<1$ and a closed $\alpha$-paraconvex set
$Y$ of a Banach space $E$, let $\mathscr{C}_{\alpha}(Y)=\{S\in
\mathscr{C}_{\alpha}(E):S\subset Y\}$ and $\mathscr{C}'_{\alpha}(Y)
=\mathscr{C}_{\alpha}(Y)\cup\{Y\}$. Since $Y$ is
$\alpha$-paraconvex, each member of $\mathscr{C}'_{\alpha}(Y)$ is
also $\alpha$-paraconvex, so it is in a good accordance with the
families $\mathscr{F}_{\alpha}(E)$ and $\mathscr{C}'_{\alpha}(E)$.
The following question was posed to the author by V. Gutev.

\begin{question}
\label{quest1}Let $X$ be a $\tau$-collectionwise normal space, $E$
be a Banach space, $0\leq\alpha<1$, and $Y$ be a nonempty
$\alpha$-paraconvex closed subset of $E$, with $w(Y)\leq\tau$. Then,
is it true that every l.s.c.\ $\varphi:X\to\mathscr{C}'_{\alpha}(Y)$
has a continuous selection?
\end{question}

To resolve Question \ref{quest2}, one can try to follow the proof of
(b) of Theorem \ref{main-theorem}. A particular difficulty to do
this is that even to take $f_0:X\to Y$ and construct $f_1$ in a
similar way, some values of $f_1$ may already go out of the set
$Y$.\medskip

For an infinite cardinal number $\tau$, a $T_1$-space $X$ is called
\emph{$\tau$-paracompact} if every open cover $\mathscr{U}$ of $X$,
with $|\mathscr{U}|\leq\tau$, has a locally finite open refinement.
In the special case of $\tau=\omega$, an $\omega$-paracompact space
is called \emph{countably paracompact}. In contrast to
paracompactness, there are $\tau$-paracompact spaces which are not
normal. Of course, a space is paracompact if and only if it is
$\tau$-paracompact for every $\tau$.\medskip

It is well known that if $X$ is $\tau$-paracompact and normal, $E$
is a Banach space, with $w(E)\leq\tau$, then every l.s.c.\
convex-valued mapping $\varphi:X\to\mathscr{F}(E)$ has a continuous
selection (see, \cite{nedev2}). Using exactly the same proof
as for the case of paracompact spaces and the above fact, one gets
the following theorem.

\begin{theorem}
Let $X$ be a $\tau$-paracompact and normal space, $E$ be a Banach
space, with $w(E)\leq\tau$, and let $\varphi:X
\to\mathscr{F}_{\alpha}(E)$ be an l.s.c.\ mapping, where $\alpha<1$.
Then, the following hold\textup{:}
\begin{itemize}
\item[(a)] $\varphi$ has a continuous selection.
\item[(b)]If $r>0$ and $g:X\to E$ is continuous such that
$d(g(x),\varphi(x))<r$ for all $x\in X$, then there exists
$\delta>0$ and a continuous selection $f$ for $\varphi$ such that
$d(g(x),f(x))<\delta r$, $x\in X$.
\end{itemize}
\end{theorem}

In the special case of $\tau=\omega$, the above theorem implies the
following consequence.

\begin{corollary}
Let $X$ be a countably paracompact and normal space, $E$ be a
separable Banach space, and let $\varphi:X
\to\mathscr{F}_{\alpha}(E)$ be an l.s.c.\ mapping, where $\alpha<1$.
Then, the following hold\textup{:}
\begin{itemize}
\item[(a)] $\varphi$ has a continuous selection.
\item[(b)]If $r>0$ and $g:X\to E$ is continuous such that
$d(g(x),\varphi(x))<r$ for all $x\in X$, then there exists
$\delta>0$ and a continuous selection $f$ for $\varphi$ such that
$d(g(x),f(x))<\delta r$, $x\in X$.
\end{itemize}
\end{corollary}

Note that in Theorems \ref{m-sel-thm} and \ref{main-theorem},
$\alpha$ is a fixed constant. Regarding this, the following question
is naturally raised:  do both theorems remain true if to each $x\in
X$, there corresponds an $\alpha(x)<1$ (possibly different for
different $x$) for which $\varphi(x)$ is $\alpha(x)$-paraconvex? A
first attempt in answering the above question was proposed by P.
Semenov \cite{semenov}, who generalized \cite[Theorem 2.1]{michael2}
by replacing the constant $\alpha$ by a function
$h:(0,+\infty)\to[0,1)$ satisfying a certain property (PS).

\begin{theorem}[\cite{semenov}]
\label{funct-paraconv-gen1} Suppose that a function
$h:(0,+\infty)\to[0,1)$ has property \textup{(PS)}, $X$ is a paracompact
space, and $E$ is a Banach space. Then, every l.s.c.\ mapping
$\varphi:X\to\mathscr{F}(E)$ whose values are $h$-paraconvex has a
continuous selection.
\end{theorem}

Here, for an arbitrary function $H:(0,\infty)\to[0,1)$, a functional sequence
$\{H_n:n<\omega\}$ is defined such that
\[
H_0(t)=1;\ \text{and}\ H_{n+1}(t)=H(H_n(t)t)\cdot H_n(t),\ n<\omega.
\]

A function $h:(0,+\infty)\to[0,1)$ has property (PS) if there is
a function $H:(0,+\infty)\to[0,1)$ strictly dominating $h$ and
such that the series $\sum_{n=0}^{\infty}H_n(t)$ converges for
all $t>0$. For a function $h:(0,+\infty)\to[0,1)$, a closed nonempty
subset $P$ of a Banach space $(E,d)$ is called \emph{$h$-paraconvex}
if for any open ball $B$ of radius $r$ that intercepts the set $P$
and for any point $q\in\overline{\conv(P\cap B)}$, then
$d(q,P)\leq h(r)r$. A closed nonempty subset of a Banach space is
said to be \emph{functionally paraconvex} if it is $h$-paraconvex
for some function $h:(0,+\infty)\to[0,1)$ (see, \cite{semenov}).
Using the technique in the proof of Theorem \ref{main-theorem}, the
following result is easily proved.

\begin{theorem}
\label{funct-paraconv-gen2}Suppose that a function $h:(0,\infty)\to
[0,1)$ has property \textup{(PS)}, $X$ is a $\tau$-collectionwise
normal space, and $E$ is a Banach space with $w(E)\leq\tau$. Then,
every l.s.c.\ mapping $\varphi:X\to\mathscr{C}'(E)$ whose values
are $h$-paraconvex has a continuous selection.
\end{theorem}

Note that if the function $h$ is equal to a constant
$\alpha<1$, then $h$-paraconvexity is equivalent to
$\alpha$-paraconvexity; and Theorem \ref{funct-paraconv-gen1}
obviously implies Theorem \ref{m-sel-thm}, while Theorem
\ref{main-theorem} is a consequence of Theorem
\ref{funct-paraconv-gen2}.\medskip

Theorem \ref{m-sel-thm} remains true for arbitrary domain $X$,
provided that the continuity of $\varphi:X\to\mathscr{F}_{\alpha}(E)$
is strengthened to $d$-continuity; that is, the following holds.

\begin{theorem}
Let $X$ be a topological space, $E$ be a Banach space, and
$\varphi:X\to\mathscr{F}_{\alpha}(E)$ be a $d$-continuous mapping,
for some $0\leq\alpha<1$. Then, $\varphi$ has a continuous selection.
\end{theorem}

It is unclear whether the above theorem holds when one further
relaxes the continuity of the mapping $\varphi$ to $d$-proximal
continuity. The following question was posed to the author by
V. Gutev.

\begin{question}
\label{quest2} Let $X$ be a topological space, $E$ be a Banach
space, and $\varphi:X\to\mathscr{F}_{\alpha}(E)$ be a $d$-proximal
continuous mapping, for some $0\leq\alpha<1$. Then, is it true that
$\varphi$ has a continuous selection?
\end{question}

Question \ref{quest2} is open even for the special case of
continuous set-valued mappings.
\vspace{1.5cm}

The author would like to express his deep gratitude to Professor V.
Gutev for introducing him to this topic and guiding him in the
preparation of this paper. The author would like also to thank
the referee for his valuable comments and suggestions.

\vspace{1cm}

\newcommand{\noopsort}[1]{} \newcommand{\singleletter}[1]{#1}
\providecommand{\bysame}{\leavevmode\hbox to3em{\hrulefill}\thinspace}
\providecommand{\MR}{\relax\ifhmode\unskip\space\fi MR }
\providecommand{\MRhref}[2]{%
  \href{http://www.ams.org/mathscinet-getitem?mr=#1}{#2}
}
\providecommand{\href}[2]{#2}


\begin{thebibliography}{1}

\bibitem{choban-valov}
Choban, M. and Valov, V. \emph{On a theorem of E. Michael on
selections}, C. R. Acad. Bulgare Sci. \textbf{28} (1975), 871--873
(in Russian).

\bibitem{dowker}
Dowker, C.H. \emph{On a Theorem of Hanner}, Ark. Mat. \textbf{2}
(1952), 307--313.

\bibitem{Engelking}
Engelking, R. \emph{General Topology, Revised and completed
edition}, Sigma Series in Pure Mathematics; \textbf{6}, Heldermann
Verag Berlin, 1989.

\bibitem{gutev}
Gutev, V. \emph{Weak Factorizations of Continuous Set-Valued
Mappings}, Topology Appl. \textbf{102} (2000), 33--51.

\bibitem{gutev-loufouma}
Gutev, V. and Loufouma Makala, NR. \emph{Selections, Extensions and
Collectionwise Normality}, J. Math. Anal. Appl. \textbf{368} (2010),
573--577.

\bibitem{gutev-al}
Gutev, V., Ohta, H., and Yamazaki, K. \emph{Selections and
Sandwich-like Properties via Semi-Continuous Banach-valued
Functions}, J. Math. Soc. Japan \textbf{55} (2003), no. 2, 499-521.

\bibitem{klee}
Klee, V. \emph{Circumspheres and Inner Products}, Math. Scand.
\textbf{8} (1960), 363--370.

\bibitem{michael}
Michael, E. \emph{Continuous Selections I}, Ann. of Math. (2), \textbf{63}
(1956), 361--382.

\bibitem{michael6}
\bysame, \emph{Continuous Selections II}, Ann. of Math. (2), \textbf{64}
(1956), 562--580.

\bibitem{michael3}
\bysame, \emph{Continuous Selection III}, Ann. of Math. (2), \textbf{65}
(1957), 375--390.

\bibitem{michael2}
\bysame, \emph{Paraconvex Sets}, Math. Scand. \textbf{7} (1959),
372--376.

\bibitem{nedev2}
Nedev, S. \emph{Four Theorems of E. Michael on Continuous Selections},
Acad\'{e}mie Bulgare des Sciences, Bulletin de l'Institut de Math\'{e}matiques,
\textbf{15} (1974), 389--393 (in Russian).

\bibitem{nedev}
\bysame, \emph{Selections and factorization theorems for set-valued
mappings}, Serdica \textbf{6} (1980), 291--317.

\bibitem{semenov}
Semenov, Pavel V. \emph{Functionally Paraconvex Sets}, Math. Notes,
\textbf{54}, No. 6 (1993).
\end{thebibliography}
\end{document}